\documentclass[11pt]{amsart}
\usepackage{latexsym}
\usepackage{amsfonts}  

\usepackage{amsmath,amsthm,amssymb}

\newtheorem{thm}{Theorem}[section]
\newtheorem{lem}[thm]{Lemma}
\newtheorem{cor}[thm]{Corollary}

\newtheorem{prop}[thm]{Proposition}
\theoremstyle{definition}
\newtheorem{defn}[thm]{Definition}
\newtheorem{notation}[thm]{Notation}

\newtheorem{rem}[thm]{Remark}

\begin{document}

\author[Ilya Kapovich]{Ilya Kapovich}

\address{\tt Department of Mathematics, University of Illinois at
  Urbana-Champaign, 1409 West Green Street, Urbana, IL 61801, USA
  \newline http://www.math.uiuc.edu/\~{}kapovich/} \email{\tt
  kapovich@math.uiuc.edu}

\author[Paul Schupp]{Paul Schupp}

\address{\tt Department of Mathematics, University of Illinois at
  Urbana-Champaign, 1409 West Green Street, Urbana, IL 61801, USA}
  \email{\tt schupp@math.uiuc.edu}

\thanks{Both authors were supported by the NSF grant DMS\#0404991. The
first author was also supported by the NSF grant DMS\#0603921}

\title[On group-theoretic models of randomness and genericity]{On group-theoretic
models of randomness and genericity}

\begin{abstract}
 We compare Gromov's density model of random groups
 with the Arzhantseva-Ol'shanskii model of genericity.
\end{abstract}



\maketitle


\section{Introduction}\label{intro}

The idea of genericity in geometric group theory was suggested by
Gromov and Ol'shanskii in late 1980s.  This theme has become the subject of active study in recent years.

The first mention of the idea of group-theoretic genericity seems to
have have been made in a 1986 paper of Guba~\cite{Gu}.
The first definition of genericity in the context of
finitely presented groups is due to Gromov and appeared in his seminal
1987 monograph ``Hyperbolic groups''~\cite{Grom}. There Gromov stated
that for any fixed $k\ge 2$ and $m\ge 1$ we have
\[
\lim_{\min n_i \to\infty}\frac{N_h(k,m, n_1,\dots, n_m)}{N(k,m, n_1,\dots,
  n_m)}=1,
\]
Here $N(k,m, n_1,\dots, n_m)$ is the number of all finite
presentations of the form
\[
\langle a_1,\dots, a_k| r_1,\dots, r_m\rangle\tag{$\ast$}
\]
where $r_i$ are cyclically reduced words with $|r_i|=n_i$ for $i=1,\dots, m$
and  $N_h(k,m, n_1,\dots, n_m)$ is the number of those among such
presentations that define word-hyperbolic groups.
Later Ol'shanskii~\cite{Ol92} and Champetier~\cite{Ch91,Ch95} gave rigorous proofs of this result.

The second model of genericity, which we term the
\emph{Arzhantseva-Ol'shanskii model}, suggested by Ol'shanskii in
1989 in a problem that appeared in the 11-th edition of Kourovka
Notebook~\cite{Kou90} (problem 11.75 in \cite{Kou90} contains a notion
that is very similar to, but slightly different from, the definition of
Arzhantseva-Ol'shanskii genericity used in \cite{A1,AO,KS} and in the
present paper). 

A property $\mathcal P$ of finitely presented groups is \emph{generic}
in the Arzhantseva-Olshanskii model (correspondingly,
\emph{exponentially generic} if the convergence to $1$ in the limit
below is exponentially fast) if for every $k\ge 2, m\ge 1$ we have
\[
\lim_{n\to\infty}\frac{\beta_\mathcal P(k,m, n)}{\beta(k,m, n)}=1,
\]
Here $\beta(k,m, n)$ is the number of presentations of the form $(\ast)$
where $\max_i |r_i|\le n$ and $\beta_\mathcal P(k,m, n)$ is the number
of such presentations that define a group with property $\mathcal P$.
We will give more precise definitions related to this model in Section~\ref{sect:ao}
below. The Arzhantseva-Ol'shanskii model is somewhat easier to work with than Gromov's
original model, since one can essentially disregard the situation
where some defining relators are much shorter than others.
This second model of genericity was formally introduced by Arzhantseva
and Ol'shanskii~\cite{AO} where they proved that the property of a
$k$-generated $m$-related group to have all $(k-1)$-generated
subgroups being free, is exponentially generic for every $k\ge 2, m\ge
1$. The Arzhantseva-Ol'shanskii model was subsequently used  by
Arzhantseva~\cite{A1,A2,A3,AC} and, later, by the authors of this
paper~\cite{KS,KSn,KSS,KS3}. For example,   Kapovich, Schupp and
Shpilrain~\cite{KSS} discovered a phenomenon of Mostow-type
isomorphism rigidity for generic one-relator groups using the
Arzhantseva-Ol'shanskii model.

In his book ``Asymptotic invariants of infinite groups''~\cite{Grom1},
Gromov introduced another model of genericity, that we refer to as
\emph{Gromov's density model} of random groups.
In this  model one first
\emph{fixes} a density parameter $0<d<1$. Then, given a  number of
generators $k\ge 2$ and an integer $n>>1$, from the set of all
cyclically reduced words of length $n$ in $F(a_1,\dots, a_k)$ one
chooses uniformly randomly and independently $(2k-1)^{dn}$ elements
forming a set $R$. Here $F(a_1,\dots, a_k)$ denotes the free group
with free basis $\{a_1,\dots, a_k\}$. The group
\[
G=\langle a_1,\dots, a_k|\ R \ \rangle
\]
is termed a \emph{random group with density parameter $d$} or a 
\emph{$d$-random group}. One then tries to understand the properties
of $G$ as $n\to\infty$. Note that the number of defining relators
$(2k-1)^{dn}$ grows exponentially in the length $n$ of the relators.
Also, crucially, the density parameter $d$ does not depend on the
number of generators $k$ of $G$.
Gromov's density model was further explored by
Ollivier~\cite{Olliv,Olliv1,Olliv2,Olliv3}, Zuk~\cite{Z} (who used a
"triangulated" variation of this model), Ollivier-Wise~\cite{OW,OW1}, and others.
Thus Ollivier~\cite{Olliv,Olliv1,Olliv4} gave a precise proof (with some generalizations to the
case of random quotients of word-hyperbolic groups) of a result first
outlined by Gromov that for $d<1/2$ a $d$-random group is
non-elementary torsion-free word-hyperbolic and for $d>1/2$ a $d$-random group is
finite (in fact either trivial or cyclic of order two).

In~\cite{Grom2} Gromov used yet another model of
randomness, which one might call a \emph{random graphical quotient model},
to prove  the existence of a finitely generated group that does not
admit a uniform embedding into a Hilbert space. 
Ghys~\cite{Gh} gives an exposition of the results and ideas
related to Gromov's density and graphical models of genericity.
A subsequent survey of Ollivier~\cite{Olliv4} gives a more updated
presentation of these topics.

Yet another approach to genericity involves considering the space
$\mathfrak S_k$ of \emph{marked groups} (that is,  the space of normal
subgroups $N$ in a fixed free group $F(a_1,\dots, a_k)$ or,
equivalently, the quotients $F(a_1,\dots, a_k)/N$) as a topological
space. One can then take the closure of some interesting class of
finitely presented groups (e.g. of word-hyperbolic groups) and try to
understand the algebraic properties of typical groups contained in
this closure. This approach was explored, in particular,
by Champetier~\cite{Ch00}.

  Our goal in this paper is  to clarify the relationship between
  Gromov's density model and the Arzhantseva-Olshanskii model.
While there is no direct connection between them,
it turns out that  proofs using  the Arzhantseva-Ol'shanskii
model often imply that a certain variation of Gromov's density
randomness condition holds.

For the purposes of comparison we need to introduce a variant of
Gromov's density model of randomness where the density parameter $d=d(k)$
depends on the number of generators $k$ and where it is possible that
$d(k)\to 0$ as $k\to\infty$.  We call this notion \emph{low-density
  randomness} (see Section~\ref{sect:random} for
precise definitions, including the definition of a \emph{monotone low-density random} property). We show in Theorem~\ref{thm:alg} that
many algebraic genericity results obtained in the
Arzhantseva-Ol'shanskii model do yield low-density random
properties:

\begin{thm}
  The following properties are monotone low-density random (where $k$
  varies over $k=2,3,\dots$):
  \begin{enumerate}
  \item~\cite{AO,A1} the property that a finite group
    presentation defines a group $G$ that is one-ended, torsion-free
    and word-hyperbolic (in fact, this property is monotone random in
    Gromov's density model~\cite{Grom2, Olliv1}).
  \item the property that a finite presentation on generators
    $a_1,\dots, a_k$ defines a group $G$ such that all
    $(k-1)$-generated subgroups are free and quasiconvex in $G$;
  \item the property that a finite presentation on generators
    $a_1,\dots, a_k$ defines a group $G$ with $rk(G)=k$.
  \item the property that a finite presentation on generators
    $a_1,\dots, a_k$ defines a group $G$ such that all $L_k$-generated
    subgroups of infinite index in $G$ are free and quasiconvex in $G$
    (here $L_k$ is any sequence of positive integers).
  \item \ the property that for a $k$-generated finitely presented
    group $G$ there is exactly one Nielsen-equivalence class
    of $k$-tuples of elements generating non-free subgroups.
  \end{enumerate}

\end{thm}
Recall that for a finitely generated group $G$ the \emph{rank} of $G$,
denoted $rk(G)$, is the smallest cardinality of a generating
set for $G$.

It turns out that in many cases various properties that are generic in
the Arzhantseva-Ol'shanskii model are not $d$-random in the sense of
Gromov  with $d$ independent of $k$.   Some key information for
estimating the density parameter $d$ in Gromov's model is contained in
the \emph{genericity entropy} of exponentially generic sets of
cyclically reduced words in $F(a_1,\dots, a_k)$.  The definition of
exponential genericity for subsets of $F(a_1,\dots, a_k)$ requires
that certain fractions converge to $1$ exponentially fast as
$n\to\infty$.  Genericity entropy quantifies this convergence rate.

We observe here that, unlike the  standard small cancellation conditions,
for the Arzhantseva-Ol'shanskii "non-readability condition"  the genericity
entropy depends on the number of generators $k$ and in fact
converges to $1$ as $k\to\infty$. This implies that, when translated
into  the language of Gromov's density model, various  results using the
Arzhantseva-Ol'shanskii model  yield properties that are
low-density random but which are  NOT $d$-random for any fixed $d>0$
which is independent of $k$. We prove this fact in detail (see
Corollary~\ref{cor:read} below) for the Arzhantseva-Ol'shanskii
non-$\mu$-readability condition.
We also show (see Proposition~\ref{prop:rank} below) that the
property for a finite presentation on $k$ generators to define a
group $G$ with $rk(G)=k$ is low-density random but not $d$-random
for any $d>0$ independent of $k$. The same is true (see
Corollary~\ref{cor:magnus} below) for the analog of Magnus'
Freiheitssatz, that is, for the property that for a group $G$
defined by a finite presentation on the generators $a_1,\dots, a_k$,
any proper subset of $a_1,\dots, a_k$ freely generates a free subgroup of
$G$.

We show, however, that certain results obtained in the
Arzhantseva-Ol'shanskii genericity model do yield $d$-random
properties in Gromov's sense. Thus we prove (see Theorem~\ref{thm:ML}
below):

\begin{thm}
For any \emph{fixed} integer $L\ge 2$ there is some $d_L>0$ such
that the property that all $L$-generated subgroups of infinite index
in a finitely presented group $G$ are free is monotone $d_L$-random.
\end{thm}

We also apply our results to a question of estimating from below the
number of isomorphism types
of quotients of $F(a_1,a_2,\dots, a_k)$ where the number of relators
is arbitrary and their length is bounded above by $n$. To be more
precise, let $k\ge 2$ be fixed and let $I_k(n)$ be the number of isomorphism types of groups
given by presentations of the form
\[
\langle a_1,\dots, a_k| R\rangle
\]
where $R$ is a subset of the $n$-ball in $F(a_1,\dots, a_k)$. Note
that the size of the $n$-ball in $F(a_1,\dots, a_k)$ is $\le 
(2k-1)^n$. Hence the number of all subsets of this ball is $\le
2^{(2k-1)^n}$ yielding a double-exponential upper bound on $I_k(n)$
as $n\to\infty$. It is natural to ask if there is also a
double-exponential lower bound for $I_k(n)$. This question was
suggested to the authors by Gromov, who informed us that several years
ago Anna Erschler obtained an unpublished proof giving such a
double-exponential lower bound.
In this paper, relying on the isomorphism rigidity results for
generic quotients of the modular group that we obtained in \cite{KS3}, we 
obtain a double-exponential lower bound for $I_k(n)$. Our proof is
quite different from that of Erschler who used central extensions of
word-hyperbolic groups to estimate $I_k(n)$ from below.

Let $M=\langle a,b| a^2=b^3=1\rangle$, so that $M$ is isomorphic to
the \emph{modular group} $ PSL(2,\mathbb Z)$. We consider
finitely presented quotients of $M$ where the defining relations are
words in the alphabet $A=\{a,b,b^{-1}\}$. There are natural notions of
a reduced and a cyclically reduced word in $A^\ast$ in this setting
(see Section~\ref{sect:mod} below for details). Note that in this context every
cyclically reduced word is either a single letter or has even length
(again, see Section~\ref{sect:mod} below).
Let $\epsilon>0$ be
fixed. For an integer $t\ge 1$ let $J_\epsilon(t)$ be the number of
isomorphism types of groups given by presentations of the form
\[
G=M/\langle\langle r_1,\dots, r_m\rangle\rangle 
\]
where $m=2^{t\epsilon}$ and where each $r_i$ is a
cyclically reduced word of length $2t$ in $A^\ast$.
We prove:

\begin{thm}\label{thm:ge}
There exists $\epsilon_0>0$ such that for any $0<\epsilon\le
\epsilon_0$ there is some $\rho>1$ such that
\[
J_\epsilon(t)\ge \rho^{\rho^t} \quad \text{ for } t\to\infty,
\]
that is, the number $J_\epsilon(t)$ is bounded below
by a double-exponential function of $t$ as $t\to\infty$.
\end{thm}

Since $M$ is generated by two elements $a$ and $b$,
Theorem~\ref{thm:ge} immediately yields a double-exponential Erschler
lower bound:

\begin{cor}\label{cor:er}
The function $I_2(n)$ has a double-exponential lower bound.
(and hence the same is true for $I_k(n)$ for any fixed $k\ge 2$).
\end{cor}

We are grateful to Lior Silberman for helpful comments regarding the
behavior of Kazhdan's Property (T) with respect to Gromov's density model.
We also thank Goulnara Arzhantseva for many helpful remarks and suggestions.

\section{Gromov's density model and low-density random groups}\label{sect:random}

In this section we want to give   some precise definitions and
notation related to Gromov's density model.

\begin{notation}
  For $k\ge 2$ let $\mathcal C_k\subseteq F(a_1,\dots, a_k)$ be the
  set of all cyclically reduced words in $F(a_1,\dots,a_k)$.  If
  $\mathcal P_k\subseteq \mathcal C_k$, we denote $\overline{\mathcal
    P_k}:=\mathcal C_k-\mathcal P_k$.  For a subset $\mathcal
  Q_k\subseteq F(a_1,\dots, a_k)$ denote by $\gamma(n,\mathcal Q_k)$ the
  number of elements of length $n$ in $\mathcal Q_k$.
\end{notation}

\begin{defn}[Random groups in the density model]\label{defn:random}
  Let $\mathcal G$ be a property of finite presentations of   groups. Let
  $0<d<1$.

  We say that the property $\mathcal G$ is \emph{random with density
    parameter $d$ ({\rm or} $d$-random)} if for every $k\ge 2$
  \[
\lim_{n\to\infty}\frac{R_k(n,d,\mathcal G)}{\gamma(n,\mathcal
  C_k)^{m_n}}=1,
\]
where $m_n=(2k-1)^{dn}$ and $R_k(n,d,\mathcal G)$ is the number of all
$m_n$-tuples $(r_1,\dots, r_{m_n})$ of cyclically reduced words of
length $n$ such that the group with presentation
\[
\langle a_1,\dots, a_k | r_1,\dots, r_{m_n}\rangle
\]
has property $\mathcal G$.

We say that $\mathcal G$ is \emph{monotone $d$-random} if for every
$0<d'\le d$ the property $\mathcal G$ is $d'$-random. A property is
\emph{monotone random} of it is monotone $d$-random for some $d>0$.
\end{defn}

Note that $\gamma(n,\mathcal C_k)^{m_n}$ is exactly the number of all
presentations
\[
\langle a_1,\dots, a_k | r_1,\dots, r_{m_n}\rangle
\]
where the $r_i$ are cyclically reduced words of length $n$.

\begin{defn}[Low-density random groups]\label{defn:lowd}
  We can consider a property $\mathcal G$ of finite presentations as
  $\mathcal G=(\mathcal G_k)_{k\ge 2}$ where for every $k\ge 2$
  $\mathcal G_k$ is a property of finite group presentations on $k$
  generators $a_1,\dots, a_k$.

  For every integer $k\ge 2$ let $0<d(k)<1$.  We say that $\mathcal G$
  is \emph{low-density random with density sequence $\left(d(k)\right)_{k\ge 2}$}
  if for every integer $k\ge 2$ we have
  \[
\lim_{n\to\infty}\frac{R_k(n,d(k),\mathcal G_k)}{\gamma(n,\mathcal
  C_k)^{m_n}}=1,
\]
were $m_n=(2k-1)^{nd(k)}$ and $R_k(n,d(k),\mathcal G_k)$ is the number
of all $m_n$-tuples $(r_1,\dots, r_{m_n})$ of cyclically reduced words
of length $n$ such that the group
\[
\langle a_1,\dots, a_k | r_1,\dots, r_{m_n}\rangle
\]
has property $\mathcal G_k$.

We say that $\mathcal G$ is \emph{monotone low-density random with
  density sequence $\left(d(k)\right)_{k\ge 2}$} if for any sequence $\left(d'(k)\right)_{k\ge 2}$ satisfying
$0<d'(k)\le d(k)$ the property $\mathcal G$ is low-density random
with density sequence $\left(d'(k)\right)_{k\ge 2}$.
\end{defn}

\begin{rem}\label{rem:0}
  In the above definition let $d:=\inf_k d(k)$ and let $\mathcal G$ be
  monotone low-density random with density sequence $\left(d(k)\right)_{k\ge 2}$. If $d>0$ then
  $\mathcal G$ is monotone $d$-random in the sense of Definition~\ref{defn:random}.

  The situation where $d=0$ does not, however,  correspond to
  a special case of Definition~\ref{defn:random}.
\end{rem}

Note that if $\mathcal G$ is a monotone low-density random property
and $\mathcal G'$ is a monotone random property with a density
parameter $d>0$ independent of $k$ then $\mathcal G\cap\mathcal G'$ is
again monotone low-density random. Moreover, the intersection of two
monotone low-density random properties is also monotone low-density
random.

In this paper we concentrate on monotone random and monotone
low-density random properties.  There are, however,  important examples
of non-monotone random properties. Thus it follows from the result of
Zuk~\cite{Z} that Kazhdan's Property (T) is $d$-random for every
$1/3<d<1/2$ (Zuk uses a somewhat different density model in
his paper but his results imply the above statement in Gromov's
density model). On the other hand, Ollivier and Wise~\cite{OW} proved that if
$0<d<1/5$ and $G$ is a $d$-random group then $G$ does not have
Property (T).

\section{The Arzhantseva-Ol'shanskii genericity model}\label{sect:ao}

We recall the basic notion of genericity in the
Arzhantseva-Ol'shanskii approach.
\begin{defn}[Generic subsets]
Let $k\ge 2$ be an integer.  A subset $\mathcal P_k\subseteq \mathcal C_k$ is \emph{generic} if
\[
\lim_{n\to\infty} \frac{\gamma(n,\mathcal P_k)}{\gamma(n,\mathcal
  C_k)}=1.
\]
We say that $\mathcal P_k\subseteq \mathcal C_k$ is
\emph{exponentially generic} if it is generic and, in addition,
the convergence to $1$ in the above limit is exponentially fast, that is, there exist $a>0$
and $0<\sigma<1$ such that for all $n\ge 1$
\[
\frac{\gamma(n,\overline{\mathcal P_k})}{\gamma(n,\mathcal C_k)}\le a
\sigma^n.
\]
This condition is equivalent to the fact that for some $0<t<1$ and
some $c>0$ we have:
\[
\gamma(n,\overline{\mathcal P_k})\le c(2k-1)^{tn}, \text{ for all }
n\ge 1.\tag{$\dag$}
\]
\end{defn}

It is not hard to show~\cite{KMSS} that a subset $\mathcal
P_k\subseteq \mathcal C_k$ is exponentially generic if and only if
\[
\lim_{n\to\infty} \frac{\#\{w\in \mathcal P_k: |w|\le n\}}{\#\{w\in\mathcal
  C_k: |w|\le n\}}=1.
\]
with exponentially fast convergence.

\begin{defn}
Let $k\ge 2$ and $m\ge 1$ be integers. We say that a subset
$\mathcal U_{k,m}\subseteq \mathcal C_k^m$ is \emph{generic} if
\[
\lim_{n\to\infty}\frac{\#\{(r_1,\dots,r_m)\in \mathcal U_{k,m}:
  |r_i|\le n, i=1,\dots, m\}}{\#\{(r_1,\dots,r_m)\in \mathcal C_k^m:
  |r_i|\le n, i=1,\dots, m\}}=1.
\]
If, in addition, this convergence is exponentially fast, we say that
$\mathcal U_{k,m}\subseteq \mathcal C_k^m$ is \emph{exponentially generic}.
\end{defn}

It is easy to see that if $\mathcal P_k\subseteq \mathcal C_k$ is
exponentially generic in $\mathcal C_k$, then for every $m\ge 1$ the
subset $\mathcal P_k^m\subseteq \mathcal C_k^m$ is exponentially
generic in $\mathcal C_k^m$. Moreover, it is also not hard to show
that in this case for every $m\ge 1$

\[
\lim_{n\to\infty}\frac{\#\{(r_1,\dots,r_m)\in \mathcal {\mathcal P}_k^m:
  |r_i|=n, i=1,\dots, m\}}{\#\{(r_1,\dots,r_m)\in \mathcal C_k^m:
  |r_i|=n, i=1,\dots, m\}}=1,
\]
with exponentially fast convergence.

\begin{defn}[Arzhantseva-Ol'shanskii genericity]
Let $\mathcal P$ be a property of groups.

For integers $k\ge 2, m\ge 1$ we say that a property of groups
$\mathcal P$ is \emph{(exponentially) $(k,m)$-generic} if the set
$\mathcal U_{k,m}$ of all $m$-tuples $(r_1,\dots, r_m)\in \mathcal C_k^m$ such
that the group $\langle a_1,\dots, a_k| r_1,\dots, r_m\rangle$
has  property $\mathcal P$, is an
(exponentially) generic subset of  $\mathcal C_k^m$. 

We say that
$\mathcal P$ is \emph{(exponentially) generic} if it is
(exponentially) $(k,m)$-generic for every $k\ge 2$, $m\ge 1$.
\end{defn}

\section{Genericity entropy and low-density randomness}

\begin{defn}
  Let $\mathcal P_k\subseteq \mathcal C_k$ be a set of cyclically
  reduced words. We define the \emph{genericity entropy} $t=t(\mathcal
  P_k)$ of $\mathcal P_k$ as:
\[
t:=\limsup_{n\to\infty} \frac{\log \gamma(n, \overline{\mathcal
    P_k})}{n\log (2k-1)}.
\]

We also define the \emph{lower genericity entropy} $t'=t'(\mathcal
P_k)$ as
\[
t':=\liminf_{n\to\infty} \frac{\log \gamma(n, \overline{\mathcal
    P_k})}{n\log (2k-1)}.
\]
\end{defn}

It is easy to see that we always have $0\le t'(\mathcal P_k)\le
t(\mathcal P_k)\le 1$ and that $\mathcal P_k\subseteq \mathcal C_k$ is
exponentially generic if and only if $t(\mathcal P_k)<1$.

A simple but crucial computation shows that  genericity entropy
controls the density parameter in Gromov's  model of random
groups:

\begin{prop}\label{prop:key}
  Let $k\ge 2$ and let $\mathcal P_k\subseteq \mathcal C_k$.

\begin{enumerate}
\item Suppose that $t:=t(\mathcal P_k)<1$.  Let $0<d<1$ be such that
  $d<1-t$.  Then:

\[
\lim_{n\to \infty} \frac{\# \text{ $(2k-1)^{dn}$-tuples of elements of
    $\mathcal P_k$ of length $n$}}{\# \text{ $(2k-1)^{dn}$-tuples of
    elements of $\mathcal C_k$ of length $n$}}=1.
\]

\item Suppose that $d>1-t'$ where $t'=t'(\mathcal P_k)$.

  Then
\[
\lim_{n\to \infty} \frac{\# \text{ $(2k-1)^{dn}$-tuples of elements of
    $\mathcal P_k$ of length $n$}}{\# \text{ $(2k-1)^{dn}$-tuples of
    elements of $\mathcal C_k$ of length $n$}}=0.
\]
\end{enumerate}

\end{prop}
\begin{proof}

  (1) Recall that there exist $0<c_0<c_1<\infty$ such that for every
  $n\ge 1$ we have
\[
c_0(2k-1)^n\le \gamma(n,\mathcal C_k)\le c_1(2k-1)^n.
\]

  Indeed, a result of Rivin ~\cite{Riv} shows that

\[  \gamma(n,\mathcal C_k) = (2k -1)^n + 1 + (k-1)[ 1 + (-1)^n] \]

  Thus for a fixed $k \ge 2$ we have $ \gamma(n,\mathcal C_k) \sim (2k - 1)^n$
where $f(n)\sim g(n)$ means that $\lim_{n\to\infty}
\frac{f(n)}{g(n)}=1$.

Let $m=(2k-1)^{dn}$. The number $N$ of $m$-tuples of elements of
$\mathcal C_k$ of length $n$ where at least one element does not belong
to $\mathcal P_k$ satisfies
\begin{gather*}
  \frac{N}{\gamma(n,\mathcal C_k)^m}\le \frac{m
    \gamma(n,\overline{\mathcal P_k})\gamma(n, \mathcal
    C_k)^{m-1}}{\gamma(n,\mathcal C_k)^m}=\frac{m \gamma(n,
    \overline{\mathcal P_k})}{\gamma(n,\mathcal C_k)}\le\\
  \le \frac{(2k-1)^{dn} c (2k-1)^{tn}}{c_0
    (2k-1)^n}=\frac{c(2k-1)^{(t+d)n}}{c_0(2k-1)^n}\to_{n\to\infty} 0.
\end{gather*}
This implies part (1) of the proposition.

(2) Again let $m=(2k-1)^{nd}$. Recall that $d>1-t'$, so that $t'>1-d$.
Let $t''$ be such that $t'>t''>1-d$. Then for $n>>1$ we have
\[
\gamma(n,\overline{\mathcal P_k})\ge (2k-1)^{nt''}
\]
and hence
\[
\gamma(n, \mathcal P_k)=\gamma(n,\mathcal C_k)-\gamma(n,
\overline{\mathcal P_k})\le \gamma(n,\mathcal C_k)-(2k-1)^{nt''}.
\]

Thus $\gamma(n,\mathcal P_k)^m$ is the number of $m$-tuples of
elements of $\mathcal P_k$ of length $n$ and it satisfies:

\begin{gather*}
  \frac{\gamma(n,\mathcal P_k)^m}{\gamma(n,\mathcal C_k)^m}\le
  \frac{(\gamma(n,\mathcal
    C_k)-(2k-1)^{nt''})^m}{\gamma(n,\mathcal C_k)^m}=\\
  \left( \frac{\gamma(n,\mathcal C_k)-(2k-1)^{nt''}}{\gamma(n,\mathcal
      C_k)}\right)^m=\left( 1-\frac{(2k-1)^{nt''}}{\gamma(n,\mathcal
      C_k)}\right)^m
\end{gather*}

Denote $Y_n=\log \frac{\gamma(n,\mathcal P_k)^m}{\gamma(n,\mathcal
  C_k)^m}$. Then

\begin{gather*}
  Y_n\le m \log \left( 1-\frac{(2k-1)^{nt''}}{\gamma(n,\mathcal
      C_k)}\right)=(2k-1)^{nd}\log \left(
    1-\frac{(2k-1)^{nt''}}{\gamma(n,\mathcal C_k)} \right)\sim\\
  (2k-1)^{nd} \left( -\frac{(2k-1)^{nt''}}{\gamma(n,\mathcal
      C_k)}\right)\sim
  -(2k-1)^{nd}\frac{(2k-1)^{nt''}}{(2k-1)^n}=\\
   =-\frac{(2k-1)^{n(d +t'')}}{(2k-1)^n} =- \left(\frac{(2k-1)^{d+t''}}{2k-1}\right)^n\to_{n\to\infty}
  -\infty,
\end{gather*}

since $d+t''>1$. (Recall that $f(n)\sim g(n)$ means that $\lim_{n\to\infty}
\frac{f(n)}{g(n)}=1$.)

Hence $\lim_{n\to\infty} \log \frac{\gamma(n,\mathcal
P_k)^m}{\gamma(n,\mathcal
  C_k)^m}=-\infty$ and therefore $\lim_{n\to\infty}
\frac{\gamma(n,\mathcal P_k)^m}{\gamma(n,\mathcal C_k)^m}=0$, as
claimed.

\end{proof}

\begin{cor}\label{cor:key}
For each $k\ge 2$ let $\mathcal P_k\subseteq \mathcal C_k$.
Let $\mathcal G =(\mathcal G_k)_{k\ge 2}$ where $\mathcal G_k$ is the
property that for a finite presentation on $k$ generators all the
defining relations belong to $\mathcal P_k$. Let $t_k=t(\mathcal P_k)$
and let $t_k'=t'(\mathcal P_k)$. Then the following hold:

\begin{enumerate}
\item If $0\le t_k<1$ for every $k\ge 2$ then the property $\mathcal
  G$ is monotone low-density random.

\item If $\sup_k t_k'=1$ then there does not exist $d>0$ such that
  $\mathcal G$ is $d$-random.
\end{enumerate}
\end{cor}

\section{Comparing the two  models}

The proofs of most existing results related to the
Arzhantseva-Ol'shanskii genericity model rely on proving that certain
subsets $\mathcal P_k\subseteq \mathcal C_k$ are exponentially
generic:

We recall the definitions of crucial genericity conditions for
many results using the Arzhantseva-Ol'shanskii genericity model.

\begin{defn}~\cite{AO} Let $0<\mu<1$ and let $k\ge 2$ be an integer. A
  freely reduced word $w$ in $F(a_1,\dots, a_k)$ is
  \emph{$\mu$-readable} if there exists a finite connected graph
  $\Gamma$, with a distinguished base-vertex, with the following properties:

\begin{enumerate}
\item Every edge $e$ of $\Gamma$ is labelled by some element $s(e)$ of
  $\{a_1,\dots, a_k\}^{\pm 1}$ so that for every edge $e$ we have
  $s(e^{-1})=s(e)^{-1}$.
\item The graph $\Gamma$ is \emph{folded} that is, there is no vertex
  with two distinct edges originating at that vertex and having the
  same label. 
\item $\Gamma$ has no degree-one vertices except possibly for
its base-vertex.
\item The fundamental group of $\Gamma$ is free of rank at most
$k-1$.
\item There exists an immersed path in $\Gamma$ labelled $w$.
\item The volume of $\Gamma$ (that is, the number of non-oriented
  edges) is at most $\mu |w|$.
\end{enumerate}
We denote by  $\mathcal P_k(\mu)$ the set of of all non-$\mu$-readable
elements of $\mathcal C_k$.

\end{defn}

\begin{defn}\label{defn:ML}~\cite{A1}
Let $L\ge 2$ and $k\ge 2$ be integers. Let $0<\mu<1$. We say that a
freely reduced word $v\in F(a_1,\dots, a_k)$ is
\emph{$(\mu,L)$-readable} if there exists  a finite connected graph
  $\Gamma$ with the following properties:

\begin{enumerate}
\item Every edge $e$ of $\Gamma$ is labelled by some element $s(e)$ of
  $\{a_1,\dots, a_k\}^{\pm 1}$ so that for every edge $e$ we have
  $s(e^{-1})=s(e)^{-1}$.
\item The graph $\Gamma$ is folded.
\item The fundamental group of $\Gamma$ is free of rank at most
$L$.
\item The graph $\Gamma$ has at least one vertex of degree $<2k$.
\item The graph $\Gamma$ has at most two degree-1 vertices.
\item There exists an immersed path in $\Gamma$ labelled $v$.
\item The volume of $\Gamma$ (that is, the number of non-oriented
  edges) is at most $\mu |v|$.
\end{enumerate}

We denote by $\mathcal Q_k(\mu,L)$ the set of all non-$(\mu,L)$-readable
elements of $\mathcal C_k$.

\end{defn}

A key result of~\cite{AO} is that for fixed $k$ and a sufficiently
small $\mu$ (namely, when
$\mu<\log_{2k}\left(1+\frac{1}{4k-4}\right)$) the set of
non-$\mu$-readable elements is exponentially generic in
$\mathcal C_k$. Arzhantseva~\cite{A1} also obtained a similar result
regarding non-$(\mu,L)$-readable words:

\begin{prop}\label{prop:gen}\cite{AO,A1}
  Let $k\ge 2$ be an integer and let $F=F(a_1,\dots, a_k)$. Then the
  following hold:

 \begin{enumerate}
 \item Let $0<\mu<\log_{2k}\left(1+\frac{1}{4k-4}\right)<1$. Then the
   set $\mathcal P_k(\mu)$ of all non-$\mu$-readable elements of
   $\mathcal C_k$ is exponentially generic in $\mathcal C_k$.

\item Let $L\ge 2$ be and integer and  let \[0<\mu<\frac{1}{3L}\log_{2k}\left(1+\frac{1}{2(2k-1)^{3L}-2}  \right)<1.\]  Then the set
  $\mathcal Q_k(\mu,L)$ of all non-$(\mu,L)$-readable words in
  $\mathcal C_k$ is exponentially generic in $\mathcal C_k$.
\end{enumerate}

\end{prop}

In most results related to the Arzhantseva-Ol'shanskii genericity one
works with intersections of properties that either monotone low-density random
(such as conditions involving non-$\mu$-readable words and
non-$(\mu,L)$-readable words) or monotone random with some density
parameter $d>0$ independent of $k$ (such as the small cancellation
condition $C'(\lambda)$ for a fixed $0<\lambda<1$). Therefore the
resulting conditions are in fact monotone low-density random. We give here a
summary of some statements that follow from the proofs of various
known results related to Arzhantseva-Ol'shanskii genericity using
Corollary~\ref{cor:key}. Next to each item we give a reference to the
source where the corresponding statement was established in the
Arzhantseva-Olshanskii model of genericity.

\begin{thm}\label{thm:alg}
  The following properties are monotone low-density random (where $k$
  varies over $k= 2,\dots$):
\begin{enumerate}

\item~\cite{AO}  the property that a finite group presentation satisfies the $C'(\lambda)$-small cancellation condition (where
  $0<\lambda\le 1/6$ is any fixed number independent of $k$) and
  defines a group $G$ that is   one-ended, torsion-free,  and
  word-hyperbolic;

\item \cite{AO} the property that a finite presentation on generators
  $a_1,\dots, a_k$ defines a group $G$ such that all $(k-1)$-generated subgroups are free and quasiconvex in $G$;
\item \cite{AO} the property that a finite presentation on generators
  $a_1,\dots, a_k$ defines a group $G$ with $rk(G)=k$;
\item \cite{A1,A2} the property that a finite presentation on generators
  $a_1,\dots, a_k$ defines a group $G$ such that all $L_k$-generated subgroups of infinite index in $G$ are free
  and quasiconvex in $G$ (here $L_k$ is any sequence of positive
  integers);
\item \cite{KS} the property that for a $k$-generated finitely presented group
  $G$ there is exactly one Nielsen-equivalence class of $k$-tuples of
  elements generating non-free subgroups.
\end{enumerate}

\end{thm}

Regarding condition (1) in Theorem~\ref{thm:alg}, it is known that the
property of a finitely presented group to be non-elementary torsion-free
word-hyperbolic is in fact monotone random and not just low-density
random (see Theorem~2 in \cite{Olliv1}, Theorem~11 in \cite{Olliv4})). It is also known and not hard to prove that if
$0<\lambda<1$ and $0<d<\lambda/2$ then the $C'(\lambda)$ small
cancellation condition is a monotone $d$-random property (see, for example, Proposition~10 of \cite{Olliv4}).

Unlike the case of the standard small cancellation condition, the
genericity entropy $t$ for exponentially generic sets arising from the
Arzhantseva-Ol'shanskii non-readability conditions usually depends on $k$ and
in fact converges to $1$ as $k\to\infty$. This situation is different from the
standard small cancellation conditions where the genericity entropy is
easily seen to have a positive upper bound which is separated from $1$ and
independent of $k$. In Section~\ref{sect:exmp}
we establish this for the non-$\mu$-readability condition and show
that in that case $\lim_{k\to\infty} t'(\mathcal P_k(\mu))=\lim_{k\to\infty}
t(\mathcal P_k(\mu))=1$. Hence for
$0<d(k)<1-t'(\mathcal P_k)$ we have $\lim_{k\to\infty} d(k)=0$ and, in
view of Corollary~\ref{cor:key}, the notion of low-density randomness
becomes necessary.

\section{Detailed examples of low-density random but not random properties}\label{sect:exmp}

\begin{prop}\label{prop:read}
  Let $k\ge 2$ and let $0<\mu_k<1$. Let $\mathcal P_k(\mu_k)\subseteq \mathcal
  C_k$ be the set of all cyclically reduced words in $F(a_1,\dots,
  a_k)$ that are not $\mu_k$-readable. Then
\[
1\ge t(\mathcal P_k)\ge t'(\mathcal P_k)\ge\frac{\log(2k-3)}{\log(2k-1)}
\]
and hence
\[
\lim_{k\to\infty} t(\mathcal P_k)=\lim_{k\to\infty} t'(\mathcal P_k)=1.
\]
\end{prop}

\begin{proof}
Let $\Gamma$ be the wedge of $(k-1)$ loop-edges labelled by
$a_1,\dots, a_{k-1}$. Then any freely reduced word from
$F(a_1,\dots, a_{k-1})$ can be read as the label of a path in
$\Gamma$. Hence for any $w\in \mathcal C_{k-1}$ with $|w|>(k-1)/\mu_k$ the word
$w$ is $\mu_k$-readable, that is,  $w\in \overline{\mathcal P_k}$.
Thus for $n\ge 1+\mu_k^{-1}(k-1)$ we have
\[
\gamma(n,\overline{\mathcal P_k})\ge \gamma(n, \mathcal C_{k-1})\ge c
(2k-3)^n
\]
for some constant $c>0$ independent of $n$.

Therefore
\begin{gather*}
t'(\mathcal P_k)\ge \liminf_{n\to\infty} \frac{\log c(2k-3)^n}{n\log
(2k-1)}=\frac{\log(2k-3)}{\log(2k-1)},
\end{gather*}
as claimed.
\end{proof}

Part (2) of Corollary~\ref{cor:key} immediately implies:

\begin{cor}\label{cor:read}
Let $0<\mu_k<1$ for $k\ge 2$. Let $\mathcal G=(\mathcal G_k)_{k\ge
2}$, where $\mathcal G_k$ is the property that for a finite
presentation on $k$ generators all the defining relations are
non-$\mu_k$-readable.

Then $\mathcal G$ is  not $d$-random for any $d>0$.
\end{cor}

Similarly, one obtains:

\begin{cor}\label{cor:magnus}
Let $\mathcal G=(\mathcal G_k)_{k\ge 2}$, where $\mathcal G_k$ is the
property that a finite group presentation on the generators
$a_1,\dots a_k$ defines a group $G$ such that every proper subset of
$a_1,\dots a_k$ freely generates a free subgroup of $G$.

Then $\mathcal G$ is monotone low-density random but not $d$-random
for any $d>0$.
\end{cor}
\begin{proof}
The fact that $\mathcal G$ is not $d$-random for any $d>0$ follows
from part (2) of  Corollary~\ref{cor:key} by the same argument as in
the proof of Proposition~\ref{prop:read}.

It is well-known (see, for example, \cite{Olliv4}, Proposition~10)
that the $C'(\lambda)$ small cancellation condition is a monotone
$d$-random property for any $0<d<\lambda/2$. 
It is also easy to see that the set of cyclically reduced words $r$ in
$\mathcal C_k$ such that every subword of $r$ of length $|r|/6$
involves all the generators $a_1,\dots, a_k$, is exponentially
generic in $\mathcal C_k$. Let $G$ be given by a
$C'(1/6)$-presentation on the generators $a_1,\dots, a_k$ where all
the defining relations $r$ have the property that every subword of
$r$ of length $|r|/6$ involves all the generators $a_1,\dots, a_k$.
Then every proper subset of $a_1,\dots, a_k$ freely generates a
subgroup of $G$. It now follows from part (1) of
Corollary~\ref{cor:key} that $\mathcal P$ is monotone low-density
random.
\end{proof}
One can regard property $\mathcal G$ from Corollary~\ref{cor:magnus}
above as a version of Magnus'  Freiheitssatz for random
groups. An  asymptotic version of the Freiheitssatz using  another
model   introduced by Gromov ~\cite{Grom}  was obtained by
Cherix and Schaeffer \cite{Che98}.

Similar arguments to those used above yield:

\begin{prop}\label{prop:rank}
  Let $\mathcal G=(\mathcal G_k)_{k\ge 2}$ where $\mathcal G_k$ is the
  property that a finite group presentation on $a_1,\dots, a_k$
  defines a group of rank $k$. Then $\mathcal G$ is monotone
  low-density random but not $d$-random for any $d>0$.
\end{prop}

\begin{proof}
  We have already observed in Theorem~\ref{thm:alg} that
  $\mathcal G$ is monotone low-density random. Let $\mathcal
  G'=(\mathcal G_k')_{k\ge 2}$ where $\mathcal G_k'$ is the property
  that for a finite presentation on $a_1,\dots, a_k$ none of the
  defining relations are primitive in $F(a_1,\dots, a_k)$. Clearly, if
  $G=\langle a_1,\dots, a_k | r_1,\dots, r_m\rangle$ and some $r_i$ is
  a primitive element in $F(a_1,\dots, a_k)$ (that is $r_i$ belongs to
  some free basis of $F(a_1,\dots, a_k)$) then $rk(G)\le k-1$. Thus
  $\mathcal G_k\subseteq \mathcal G_k'$ and $\mathcal G\subseteq
  \mathcal G'$.  It suffices to show that $\mathcal G'$ is not
  $d$-random for any $d>0$.

  Let $\mathcal P_k\subseteq \mathcal C_k$ be the set of all
  non-primitive elements in $\mathcal C_k$.  Note that, for any freely
  reduced word $w\in F(a_2,\dots a_k)$, the element $a_1w$ is primitive
  in $F(a_1,\dots, a_k)$. Hence
\[
\gamma(n,\overline{\mathcal P_k})\ge \gamma(n-1,
F_{k-1})=(2k-2)(2k-3)^{n-2}.
\]
Therefore
\[
1\ge t(\mathcal P_k)\ge t'(\mathcal P_k)\ge
\frac{\log(2k-3)}{\log(2k-1)}\to_{k\to \infty} 1.
\]
Part (2) of Corollary~\ref{cor:key} implies that $\mathcal G'$ is
not $d$-random for any $d>0$.
\end{proof}

\section{A bounded freeness property}\label{sect:ML}

In this section we will show that for every fixed integer $L\ge 2$
there is some $0<d<1$ such that the property
of a finitely presented group that  all its $L$-generated subgroups of
infinite index are  free is monotone $d$-random.

First, we need to investigate the genericity entropy of the set of
non-$(\mu,L)$-readable words.
Recall that $\mathcal Q_k(\mu,L)$ is the set of all words in $\mathcal
C_k$ that are not $(\mu,L)$-readable.
The proof of the following proposition is similar to the counting
arguments used in \cite{A1,AO}, with a variation whose significance is
explained further in Remark~\ref{rem:H} below.

\begin{prop}\label{prop:ML}
Let $k\ge 2$ be a fixed integer and let $2\le L<k$.
Then we have:
\[
\gamma(n, \overline{\mathcal Q_k(\mu,L)})\le C  (\mu n)^{3L+1}  (6L)^n (2k-1)^{\mu n}.
\]
where $C>0$ is independent of $n$.
\end{prop}
\begin{proof}
Recall that an \emph{arc} in $\Gamma$ is an immersed edge-path where
every intermediate  vertex of the path has degree two in
$\Gamma$.

Note that if $\Gamma$ is a finite connected graph with fundamental
group free of rank $\le L<k$, then $\Gamma$ necessarily has a vertex
of degree $<2k$. Thus condition (4) of Definition~\ref{defn:ML} is
redundant in this case.

Let $L>k$ and $0<\mu<1$ be fixed.
Let $v\in F(a_1,\dots, a_k)$ be a $(\mu,L)$-readable word with $|v|=n$.

First, we estimate the number of labelled graphs $\Gamma$ as in
Definition~\ref{defn:ML} where $v$ can be read.

There are $\le C_0=C_0(L)$ topological types of the graphs $\Gamma$ arising
in the definition of a $(\mu,L)$-readable word.
Since $\pi_1(\Gamma)$ has rank at most $L$ and $\Gamma$ has
at most two degree-1 vertices, it follows that $\Gamma$ has $\le 3L$
non-directed maximal arcs and $\le 6L$ directed maximal arcs.

The sum of the length of theses arcs is $\le \mu n$.
The number of ways to represent a positive integer $N$ as a sum
\[
N=N_1+\dots+N_{3L}
\]
where $N_i$  are non-negative
integers is
\[
 \frac{ (N + 3L - 1)! }{ N! (3L -1)! }  \le (N+3L-1)^{3L}.
\]
 Hence the number of ways to write a sum
\[
N_1+\dots +N_{3L}\le \mu n
\]
is $\le C_1 (\mu n)^{3L+1}$, where $C_1$ is independent of $n$.
For each decomposition $N_1+\dots +N_{3L}\le \mu n$ the number of ways
to assign the maximal arcs of $\Gamma$ labels $v_1,\dots v_{3L}\in
F(a_1,\dots, a_k)$ with $|v_i|=N_i$ is
\[
\le C_2 (2k-1)^{\mu n}
\]
where $C_2>0$ does not depend on $n$.

Thus there are at most $C_0C_1C_2 (\mu n)^{3L+1} (2k-1)^{\mu n}$
relevant labelled graphs $\Gamma$ as in Definition~\ref{defn:ML}

For a fixed $\Gamma$, if $v$ can be read in $\Gamma$ then $v$ is the
label of a path \[p_1',p_2,\dots, p_{s-1},p_s'\] where $p_i$ are
oriented maximal arcs, $p_1',p_s'$ are oriented arcs and $s\le |v|=n$.
By passing to a subgraph of $\Gamma$ if necessary we may assume that
$p_1'$ and $p_2'$ are maximal arcs as well. Thus $v$ is the label of a
path $\alpha=p_1,p_2,\dots, p_{s-1},p_s$ where $p_i$ are directed
maximal arcs in $\Gamma$ and where $s\le n=|v|$. Since $s\le n$ and
$\Gamma$ has $\le 6L$ oriented maximal arcs, there are $\le (6L)^n$
combinatorial possibilities to express $\alpha$ as a word in the
alphabet of $6L$ letters corresponding to the directed maximal arcs.

Hence the total number of possibilities for $v$ is
\[
\gamma(n,\overline{\mathcal Q_k(\mu,L)})\le C_0C_1C_2 (\mu n)^{3L+1} (6L)^n (2k-1)^{\mu n},
\]
as required.

\end{proof}

The following technical definition is motivated by the corresponding
notions used in counting arguments in \cite{A1,AO}.

\begin{defn}
Let $k\ge 2$, $L\ge 2$ be integers and let $0<\mu<1$. We say that a
cyclically reduced word $w\in F(a_1,\dots, a_k)$ is
\emph{$(\mu,L)$-good} if no cyclic permutation of $w^{\pm 1}$ contains
a subword $v$ of length $\ge |w|/2$ such that $v$ is $(\mu,L)$-readable.
\end{defn}

\begin{lem}\label{lem:ML}
Let $k>L\ge 2$ and let $0<\mu<1$. Let $\mathcal Y_k=\mathcal
Y_k(\mu,L)\subseteq \mathcal C_k$ be the set of all cyclically reduced
$(\mu,L)$-good words.
Then
\[
t(\mathcal Y_k)\le \frac{((\mu+1)/2)\log(2k-1)+(1/2)\log(6L)}{\log (2k-1)}.
\]
\end{lem}
\begin{proof}
Let $w\in \overline{\mathcal Y_k}$ with $n=|w|$. There are at most
$2n$ cyclic permutations of $w^{\pm 1}$ and at least one of them has
an initial segment $v$ of length $n/2$ such that $v$ is
$(\mu,L)$-readable. Hence by Proposition~\ref{prop:ML} the number of possibilities for $w$ is
\[
\gamma(n, \overline{\mathcal Y_k})\le A (2n) (\mu n/2)^{3L+1} (6L)^{n/2} (2k-1)^{\mu n/2}(2k-1)^{n/2},
\]
where $A>0$ is independent of $n$.
Hence
\begin{gather*}
t(\mathcal Y_k)=\limsup_{n\to\infty} \frac{\log \gamma(n, \overline{\mathcal
    Y_k})}{n\log (2k-1)}\le\\
\limsup_{n\to\infty} \frac{(\frac{n}{2}+\frac{\mu n}{2})\log
(2k-1)+\frac{n}{2}\log
  6L+\log(2An)+(3L+1)\log\frac{\mu n}{2}}{n\log (2k-1)}=\\
=\frac{(\frac{\mu+1}{2})\log(2k-1)+\frac{1}{2}\log(6L)}{\log
(2k-1)}.
\end{gather*}
\end{proof}

The results of Section~4 of \cite{A1} imply:

\begin{prop}\label{prop:tech}
Let $L,k\ge 2$ be integers. Let $0<\mu<1$ and $0<\lambda<1$ be
such that
\[
0<\lambda\le \frac{\mu}{15L+3\mu}\le \frac{1}{6}.
\]
Let $G=\langle a_1,\dots, a_k| r_1,\dots, r_m\rangle$ be such that
\begin{enumerate}
\item The above presentation of $G$ satisfies the small cancellation condition $C'(\lambda)$.
\item All $r_1,\dots, r_m$ are cyclically reduced words that are not
  proper powers in $F(a_1,\dots, a_k)$.
\item Each $r_i$ is $(\mu, L)$-good.
\end{enumerate}
Then every $L$-generated subgroup of infinite index in $G$ is free.
\end{prop}

\begin{thm}\label{thm:ML}
For every integer $L\ge 2$ there is some $d>0$ such that the property
of finitely presented groups for all $L$-generated subgroups of
infinite index to be free is monotone $d$-random.
\end{thm}

\begin{proof}
Let $L\ge 2$ be a fixed integer.

It is well-known and easy to see that conditions (1) and (2) from
Proposition~\ref{prop:tech} are monotone random~(see, e.g. Proposition
10 and Theorem 11 in \emph{Olliv4}).
Thus it suffices to deal with condition (3) of
Proposition~\ref{prop:tech}.

Choose $0<\lambda,\mu<1$ so that
\[
0<\lambda\le \frac{\mu}{15L+3\mu}\le \frac{1}{6}.
\]

We have
\[
\lim_{k\to\infty} \frac{((\mu+1)/2)\log(2k-1)+(1/2)\log(6L)}{\log (2k-1)}=\frac{\mu+1}{2}<1.
\]
Choose $\nu$ so that $(\mu+1)/2<\nu<1$. There exists an integer
$k_0>L$ such that for any $k\ge k_0$
\[
\frac{((\mu+1)/2)\log(2k-1)+(1/2)\log(6L)}{\log (2k-1)}\le \nu<1.
\]
Thus by Lemma~\ref{lem:ML} for $k\ge k_0$ we have
\[
t(\mathcal Y_k)\le \nu<1.
\]

Recall that by Theorem~\ref{thm:alg} the property of having all
$L$-generated subgroups of infinite index being free is monotone
low-density random with density sequence $(d(k))_{k\ge 2}$.  Put
$d_0:=\min\{d(2),\dots, d(k_0-1), 1-\nu\}$.  Then by
Proposition~\ref{prop:key} the property of having all $L$-generated
subgroups of infinite index being free is monotone $d$-random for any
$0<d<d_0$.
\end{proof}

\begin{rem}\label{rem:H}

In \cite{A1} Arzhantseva gave a proof of exponential genericity in
$F(a_1,\dots, a_k)$ of non-$(\mu,L)$-readable
words, assuming that $\mu$ is small enough. However, the estimates on
the growth of $(\mu, L)$-readable words obtained there are
insufficient for our purposes in the proof of Theorem~\ref{thm:ML}.
Let $\mathcal P_k\subseteq \mathcal C_k$ be the set of all
non-$(\mu_k,L)$-readable cyclically reduced words in $F(a_1,\dots,
a_k)$, where $0<\mu_k<1$ satisfies
\[
0<\mu_k<\frac{1}{3L}\log_{2k}\left(1+\frac{1}{2(2k-1)^{3L}-2}  \right).
\]

A crucial estimate in Lemma~3 of \cite{A1} shows that
\[
\gamma(n,\overline{\mathcal P_k})\le A \left((2k-1)^{3L}-\frac{1}{2} \right)^{n/3L}.\tag{$\ast\ast$}
\]
This yields
\[
t(\mathcal P_k)\le \frac{\log\left( (2k-1)^{3L}-\frac{1}{2}
  \right)}{3L\log(2k-1)}\to_{k\to\infty} 1,
\]
where convergence to $1$ in the last limit is easily seen by applying
l'H\^opital's rule.  Therefore we needed an estimate  different from
$(\ast\ast)$  for the number of $(\mu,L)$-readable words in
Proposition~\ref{prop:ML}. That estimate allowed us to obtain bounds on the
genericity entropy of the set of $(\mu,L)$-good words that are
independent of $k$ for sufficiently large $k$. On the other hand, we
still needed the results of \cite{A1} obtained via the estimate $(\ast)$
to deal with the case of ``small'' $k$ with $k < k_0$ in the proof of
Theorem~\ref{thm:ML}.
\end{rem}

\section{Double-exponential lower bound for $J_\epsilon(t)$}\label{sect:mod}

In this section we establish Theorem~\ref{thm:ge} from the
Introduction (see Theorem~\ref{thm:count} below).
Note that Theorem~\ref{thm:ge} implies
Corollary~\ref{cor:er} giving Erschler's double-exponential lower bound for the
number $I_2(n)$ of isomorphism types of quotients of $F(a,b)$ by
collections of defining relations of length $\le n$.

It is well-known that the \emph{modular group} $\mathbb PSL(2,\mathbb
Z)$ is isomorphic to the free product of a cyclic group of order two
and a cyclic group of order three.
Denote
\[
M:=\langle a,b| a^2=b^3=1\rangle=\langle a| a^2=1\rangle \ast \langle b| b^3=1\rangle.
\]
Put $A=\{a,b,b^{-1}\}$. We say that a word $w\in A^\ast$ is
\emph{reduced} if it does not contain subwords of the form $aa$, $bb$,
$b^{-1}b^{-1}$, $bb^{-1}$, $b^{-1}b$. it is clear that any element of
$M$ is uniquely represented by a reduced word in $A^\ast$. We say that
a word $w\in A^\ast$ is \emph{cyclically reduced} if $w$ and all
cyclic permutations of $w$ are reduced. Thus any nonempty cyclically reduced
word is either a single letter or, up to a cyclic permutation, has the form
\[
w=ab^{\epsilon_1}ab^{\epsilon_2}\dots ab^{\epsilon_t}
\]
where $\epsilon_i=\pm 1$.
It is therefore easy to see that the number of all cyclically reduced
words in $A^\ast$ of length $n>1$ is equal to $0$ if $n$ is odd and is
equal to $2\cdot 2^{n/2}$ if $n$ is even. As before, let $\mathcal
C_A$ be the set of all cyclically reduced words in $A^\ast$. For a
subset $S\subseteq \mathcal C_A$ denote by $\gamma(n,S)$ the number of
elements of length $n$ in $\mathcal C_A$. Similarly to the free group
case, we can define the notions of \emph{generic} and
\emph{exponentially generic} subsets of $\mathcal C_A$. Thus
$S\subseteq \mathcal C_A$ is \emph{exponentially generic} if
\[
\lim_{t\to\infty} \frac{\gamma(2t,S)}{\gamma(2t,\mathcal C_A)}=\lim_{t\to\infty} \frac{\gamma(2t,S)}{2^{t+1}}=1
\]
with exponentially fast convergence.
Similarly, all the other notions of genericity in the Arzhantseva-Ol'shanskii model
can be defined for quotients of $M$ in exactly the same way as for the
quotients of $F(a_1,\dots, a_k)$.

Denote by $\eta:M\to M$ the \emph{relabelling automorphism} of $M$
defined on the generators as $\eta(a)=a$, $\eta(b)=b^{-1}$.

\begin{notation}
For $\epsilon>0$ be fixed. For an integer $t\ge 1$ let $J_\epsilon(t)$ be the number of
isomorphism types of groups given by presentations of the form
\[
G=M/\langle\langle r_1,\dots, r_m\rangle\rangle \tag{$\ddag$}
\]
where $m=2^{t\epsilon}$ and where each $r_i$ is a
cyclically reduced word of length $2t$ in $A^\ast$.
\end{notation}

\begin{thm}\label{thm:count}
There exists $\epsilon_0>0$ such that for any $0<\epsilon\le
\epsilon_0$ there is some $\rho>1$ such that
\[
J_\epsilon(t)\ge \rho^{\rho^t} \quad \text{ for } t\to\infty,
\]
that is,
$J_\epsilon(t)$ is bounded below
by a double-exponential function of $t$ as $t\to\infty$.
\end{thm}

\begin{proof}
The results of \cite{KS3} show that there is some exponentially
generic subset $S\subseteq A^\ast$ and some $0<\lambda<1$ with the following property.
Suppose $m\ge 1$ is fixed. Then there exists an exponentially generic
subset $U_m\subseteq \mathcal C_A^m$ such that:

(1) Every presentation $\ddag$ with $(r_1,\dots, r_m)\in U_m$
satisfies the $C'(\lambda)$ small cancellation condition.

(2) We have $U_m\subseteq S^m$.

(3) For $(r_1,\dots, r_m), (s_1,\dots, s_m)\in U_m$ with
$|r_i|=|s_j|=2t$ the groups
$M/\langle\langle r_1,\dots, r_m\rangle\rangle$ and  $M/\langle\langle
s_1,\dots, s_m\rangle\rangle$ are isomorphic if and only if there is a
reordering $(r_1',\dots, r_m')$ of $(r_1,\dots, r_m)$ and there is
$\delta\in \{0,1\}$ such that each $r_i'$ is a cyclic permutation of
$\eta^{\delta}(s_i)$ or of $\eta^{\delta}(s_i^{-1})$.

(4) The number $K_m(t)$ of isomorphism types of groups given by
presentation $(\ddag)$ where all $r_i$ are cyclically reduced words of
length $2t$ in $A^\ast$ satisfies
\[
K_m(t)\sim \frac{2^{m(t+1)}}{2 \, m! (4t)^m}
\]

Statement (4) is essentially a corollary of (3): one needs to count
the number of all presentations $(\ddag)$ where $(r_1,\dots, r_m)\in
U_m$ has $|r_1|=\dots=|r_m|=2t$ and
divide this number by the multiplicity constant in counting the isomorphism
types of such presentations, where this multiplicity constant comes
from (3) and is equal to $2 \, m! (4t)^m$. 
Here the factor $m!$ comes from counting reorderings $(r_1',\dots, r_m')$ of $(r_1,\dots, r_m)\in U_m$. 
Every $r_i$ of length $2t$ has $2t$ cyclic permutations, so there are $4t$ cyclic permutations of $r_i^{\pm 1}$.
Finally, applying $\eta^\delta$, with $\delta=0,1$, to the presentation gives an additional multiplicity factor of $2$.  

The results of the present paper, namely an appropriately adapted version of Corollary~\ref{cor:key}, imply that
statements (1)-(3) also hold in the low-density model, where the
number of relations $m$ is
not fixed but rather has the form $m=2^{t\epsilon}$,
where $\epsilon>0$ is a sufficiently small number independent of $t$.
Note that since $S\subseteq \mathcal C_A$ is exponentially generic, we
have $\gamma(2t,S)\ge \frac{1}{2}\gamma(2t,\mathcal C_A)$ for all
sufficiently large $t$. 
Then the same arguments as in \cite{KS3} imply that the number
$J_\epsilon(t)$ satisfies
\[
J_\epsilon(t)\ge C\frac{2^{m(t+1)}2^{-m}}{2 \, m! (4t)^m}
\]
where $C>0$ is a constant and where $m=2^{t\epsilon}$.
It is not hard to see that this gives a double-exponential lower bound
for $J_\epsilon(t)$. Indeed, note that $m!\le m^m$ and thus
\[
J_\epsilon(t)\ge C\frac{(2^{(t+1)})^m}{2 (8mt)^m}\ge C\frac{(2^{(t+1)})^m}{(16mt)^m}
\]
hence
\[
\log J_\epsilon(t) \ge \log C+ m\, \log\left( \frac{2^{(t+1)}}{16mt}
\right)=\log C+ 2^{t\epsilon}\, \log\left( \frac{2^{(t+1)}}{16\cdot 2^{t\epsilon}t} \right)
\]
If $\epsilon>0$ is chosen sufficiently small, then
\[
\frac{2^{(t+1)}}{16\cdot 2^{t\epsilon}t}\ge 2 \quad\text{ for } t\to\infty
\]
and hence
\[
\log J_\epsilon(t) \ge \log C+ 2^{t\epsilon}\, \log 2\ge  2^{t\epsilon/2} \quad\text{ for } t\to\infty,
\]
yielding a double-exponential lower bound for $J_\epsilon(t)$, as required.
\end{proof}

\end{document}